\documentclass{article}

\usepackage{amsthm}
\usepackage[sc]{titlesec}
\usepackage{titlecaps}
\usepackage{hyperref}
\usepackage{amsfonts}
\usepackage{MnSymbol}
\usepackage[margin=1.3in]{geometry}
\usepackage{enumitem}
\usepackage{xypic}
\usepackage{xcolor}
\usepackage{tikz}

\newtheorem{theorem}{Theorem}[section]
\newtheorem{lemma}[theorem]{Lemma}
\newtheorem{cor}[theorem]{Corollary}
\newtheorem{prop}[theorem]{Proposition}

\theoremstyle{definition}
\newtheorem{definition}[theorem]{Definition}

\theoremstyle{remark}
\newtheorem{remark}[theorem]{Remark}

\newtheoremstyle{named}{}{}{\itshape}{}{\bfseries}{.}{.5em}{#3}
\theoremstyle{named}
\newtheorem{namedtheorem}{Theorem}
\newtheorem{namedprop}{Proposition}

\begin{document}

\title{\textsc{Limit groups over coherent right-angled Artin groups are cyclic subgroup separable}}
\author{Jonathan Fruchter}
\date{}
\maketitle

\begin{abstract}
We prove that cyclic subgroup separability is preserved under exponential completion for groups that belong to a class that includes all coherent RAAGs and toral relatively hyperbolic groups; we do so by exploiting the structure of these completions as iterated free products with commuting subgroups. From this we deduce that the cyclic subgroups of limit groups over coherent RAAGs are separable, answering a question of Casals-Ruiz, Duncan and Kazachov. We also discuss relations between free products with commuting subgroups and the word problem, and recover the fact that limit groups over coherent RAAGs and toral relatively hyperbolic groups have a solvable word problem.
\end{abstract}

\tableofcontents

\section{Introduction}

In a recent paper, \cite{raags}, Casals-Ruiz, Duncan and Kazachov defined a new class of groups $\mathcal{C}$, which was carefully designed to serve as a general framework for studying limit groups over (coherent) right-angled Artin groups (RAAGs). They succeed in showing that the limit groups over any coherent right-angled Artin group $G$ embed in the $\mathbb{Z}[t]$-completion of $G$, and that this completion can be built by repeatedly extending centralisers. Casals-Ruiz, Duncan and Kazachov \cite{raags} then pose the question: are cyclic subgroups of limit groups over coherent RAAGs closed in the profinite topology? Our main purpose here is to give a positive answer to this question.
\par
\medskip 
Constructing completions of groups using extensions of centralisers is not a new idea; in the 1960s, Baumslag gave a construction of the $\mathbb{Q}$-completion of a group with unique roots, which uses extensions of centralisers. Many would argue that the solution to Tarski's question about the first order theory of non-abelian free groups epitomized the use of extensions of centralisers. As a first step towards showing that all non-abelian free groups share the same first order theory, Sela investigated the structure of limit groups (over free groups) in \cite{sela} and showed that they admit a hierarchical structure. This structure implies that limit groups embed in the $\mathbb{Z}[t]$-completion of a free group, or in other words $F^{\mathbb{Z}[t]}$ serves as a universe for limit groups over the free group $F$. Similar results were obtained by Kharlampovich and Miasnikov, and in \cite{torallim} they extended their argument to prove that the limit groups over any toral relatively hyperbolic group $G$ embed in $G^{\mathbb{Z}[t]}$ (the definition of a limit group over any group $G$ is recalled in Subsection \ref{limitgps} below).
\par 
\medskip
The latter work was carried out in the context of groups in the class $\mathrm{CSA}$: groups whose maximal abelian subgroups are malnormal. The authors showed that extending a centraliser of a $\mathrm{CSA}$ group yields a $\mathrm{CSA}$ group, which allowed them to use induction in the process of repeatedly extending centralisers. RAAGs, even coherent ones, do not necessarily lie in the class $\mathrm{CSA}$. This deficit motivated Casals-Ruiz, Duncan and Kazachov to seek a broader setting in which similarly structured proofs would work, leading them to the class $\mathcal{C}$. They show in \cite{raags} that this class $\mathcal{C}$ contains all coherent RAAGs, as well as all toral relatively hyperbolic groups. In addition, they prove that if $G$ is in the class $\mathcal{C}$ (and satisfies the technical condition $\mathrm{R}$ of Definition \ref{condR}) then $G^{\mathbb{Z}[t]}$ can be built as an iterated centraliser extension over $G$ and is fully residually $G$. They also show that if $G$ is a coherent RAAG, then every limit group over $G$ embeds in $G^{\mathbb{Z}[t]}$. They suggested that this fact, combined with the relatively simple structure of $G^{\mathbb{Z}[t]}$, ought to provide fertile ground for addressing algorithmic problems and establishing residual properties of limit groups over RAAGs, and highlighted some specific challenges of this type, which we address in this paper.
\par 
\medskip
As well as building on \cite{raags}, the results that we shall present here rely crucially on the simple observation that direct extensions of centralisers, which are the basis for the construction of $\mathbb{Z}[t]$-completions, have a much tamer structure than more general amalgamated products. By definition, if  $C_G(u)$ denotes the centraliser of $u$ in $G$, then the \emph{direct extension} of $C_G(u)$ by $C_G(u)\times B$ is the quotient of the free product $G\ast B$ by relations that force the subgroups $C_G(u)$ and $B$ to commute. This is an example of a \emph{free product with commuting subgroups}: such a product is obtained from pairs of groups $L\le G$ and $M\le H$ by forming the amalgamated free product $G\ast_L(L\times M)\ast_M H$, which is abbreviated to $\langle G, H \vert \; [L,M]=1\rangle$. We shall exploit the way in which $G^{\mathbb{Z}[t]}$ is built from extensions of centralizers to prove the following theorems, repeatedly employing a criterion developed by Loginova \cite{loginovacyc} for the cyclic subgroup separability of certain free products with commuting subgroups. When $G$ is a coherent RAAG, the special combinatorial structure of its defining graph (see Subsection \ref{raagdef}) also plays an important part in our proof.

\begin{namedtheorem}[Theorem 1]
	\label{main}
	Let $G$ be a group in the class $\mathcal{C}$ which satisfies condition $\mathrm{R}$, and let $A$ be a ring. If $G$ is cyclic subgroup separable, then the $A$-completion of $G$, $G^A$, is cyclic subgroup separable.
\end{namedtheorem}

\begin{namedtheorem}[Theorem 2]
	\label{css}
	Limit groups over coherent RAAGs are cyclic subgroup separable.
\end{namedtheorem}

In the classic (and more restrictive) setting of limit groups over free groups, Wilton showed that all finitely generated subgroups are separable \cite{wiltonhall}. However, limit groups over coherent RAAGs are not necessarily subgroup separable: the coherent RAAG given by the presentation $L=\langle x,y,z,w \vert \; [x,y]=[y,z]=[z,w]=1\rangle$, and whose defining graph is
\begin{center}
\begin{tikzpicture}
\filldraw[black] (-1.5,0) circle (2pt);
\filldraw[black] (-0.5,0) circle (2pt);
\filldraw[black] (0.5,0) circle (2pt);
\filldraw[black] (1.5,0) circle (2pt);
\draw (-1.5,0)--(1.5,0);
\end{tikzpicture}
\end{center}
is not subgroup separable \cite[Theorem 1.2]{linkgroup}.
\par 
\medskip
In the last section of the paper we shall discuss the word problem for free products with commuting subgroups. Given $L\le G$ and $M\le H$, if the word problem is solvable both in $G$ and in $H$, and the membership problem is solvable for $L$ in $G$ and for $M$ in $H$, then there is a solution to the word problem in $\langle G,H \vert \; [L,M]=1\rangle$. We shall use once again the way in which exponential completions of groups from the class $\mathcal{C}$ can be built by iterating centraliser extensions to prove the following:
\begin{namedprop}[Proposition 3]
	\label{decidableword}
	Let $G$ be a group in the class $\mathcal{C}$. If $G$ satisfies condition $\mathrm{R}$ and has a solvable word problem, then every finitely generated subgroup $H$ of $G^A$ has a solvable word problem.
\end{namedprop}
In particular, we recover the fact that limit groups over coherent RAAGs and toral relatively hyperbolic groups have a solvable word problem.
\par 
\medskip
\paragraph{Acknowledgements}
I would like to thank Martin Bridson for his warm encouragement and generous feedback.

\section{Preliminaries}

\subsection{Notations and conventions}
Given a group $G$, we denote by $Z(G)$ its center, and by $C_G(g)$ the centraliser of $g\in G$. Throughout this paper, we assume that all rings are associative, have a free abelian additive group and a multiplicative identity $1$. For such a ring $A$, the additive group generated by $1$ is denoted $\mathrm{char}(A)\cong \mathbb{Z}$ and whenever we refer to $\mathbb{Z}\subset A$ we in fact mean $\mathrm{char}(A)$. For a graph $\Gamma$, we denote by $\mathrm{V}\Gamma$ the set of vertices of $\Gamma$ and by $\mathrm{E}\Gamma \subset \mathrm{V}\Gamma \times \mathrm{V}\Gamma$ the set of edges of $\Gamma$.
\par
\medskip
We recall the notion of a \emph{free product with commuting subgroups}: let $G$ and $H$ be groups, and suppose that $L$ and $M$ are subgroups of $G$ and $H$ respectively. The \emph{free product of $G$ and $H$ with commuting subgroups $L$ and $M$} is the quotient of the free product $G \ast H$ by the normal closure of the set of relations $\{[\ell,m] \vert \; \ell \in L \text{ and } m \in M\}$. We often abbreviate and refer to this group as $\langle G,H \vert \; [L,M]=1 \rangle$.

\subsection{Subgroup separability}
A subgroup $H$ of a group $G$ is called \emph{separable} if it is an intersection of finite index subgroups of $G$. If $H$ is separable in $G$, then every $g\notin H$ can be separated from $H$ in a finite quotient of $G$, that is, there is a homomorphism $f:G\rightarrow Q$ where $Q$ is finite and $f(g)\notin f(H)$. Another elegant description of a separable subgroup is in terms of the profinite topology on $G$. Recall that in this topology, a local base of the identity in $G$ is the set of finite index normal subgroups of $G$. A local base at any $g\in G$ is then obtained by taking $g$-cosets of the finite index normal subgroups of $G$. Furthermore, every finite index normal subgroup of $G$ is also closed in this topology. Therefore, a subgroup $H\le G$ is separable if and only if it is closed in the profinite topology on $G$.

\begin{definition}
	A group $G$ is called \emph{cyclic subgroup separable} if each of its cyclic subgroups is separable.
\end{definition}

\subsection{Right-angled Artin groups and the class $\mathcal{C}$}
\label{raagdef}
We remind the reader of the definition of a right-angled Artin group:
\begin{definition}
	Let $\Gamma$ be a simple graph; the \emph{right-angled Artin group} (or in short, RAAG) $G(\Gamma)$ is the group with presentation
	\begin{equation*}
		\langle \mathrm{V}\Gamma \; \vert \;\; [v,u], \; (v,u)\in \mathrm{E}\Gamma \rangle.
	\end{equation*}
	We refer to the graph $\Gamma$ as the \emph{defining graph} of $G(\Gamma)$.
\end{definition}
Note that in the definition above we do not restrict ourselves to finite graphs. Recall that a group is called \emph{coherent} if all of its finitely generated subgroups are finitely presented. In \cite[Theorem 1]{coherent}, Droms shows that a finitely generated RAAG $G(\Gamma)$ is coherent if and only if its defining graph $\Gamma$ is \emph{chordal}: every subgraph of $\Gamma$ that is a cycle of more than $3$ vertices admits a chord, i.e., an edge that connects two vertices of the cycle. Note that if $\Gamma$ is an infinite chordal graph, then every finitely generated subgroup $H$ of $G(\Gamma)$ is a subgroup of a RAAG $G(\Gamma')$, where $\Gamma'$ is a finite and full subgraph of $\Gamma$; $\Gamma'$ is chordal, which implies that $H$ is finitely presented and therefore $G(\Gamma)$ is coherent. The class of coherent RAAGs includes free groups, free abelian groups and RAAGs which are fundamental groups of 3-manifolds (see \cite[Theorem 2]{coherent}).
\par 
\medskip
In \cite{raags}, Casals-Ruiz, Duncan and Kazachov define a new class of groups $\mathcal{C}$ which was crafted specifically to satisfy the following property: the $\mathbb{Z}[t]$-completion (see Subsection \ref{Agroup})  of a group $G$ in the class $\mathcal{C}$ can be built by iterating extensions of centralisers (see Section \ref{mainsec}), and is fully residually $G$. The definition of the class $\mathcal{C}$ is long and technical, and is beyond the scope of this paper; in this paper we do not use the definition of the class $\mathcal{C}$ directly, but rather use properties of groups in the class $\mathcal{C}$ proven in \cite{raags}. We briefly mention that a group $G$ in the class $\mathcal{C}$ is torsion-free, has unique roots and satisfies the Big Powers ($\mathrm{BP}$) property: for every $g_1,\ldots,g_k\in G$ such that $[g_i,g_{i+1}]\ne 1$, there exists a positive integer $N$ such that for every $n_1,\ldots,n_k>N$,
\begin{equation*}
	g_1^{n_1}\cdots g_k^{n_k}\ne 1.
\end{equation*}
The $\mathrm{BP}$ property is used to show that extending centralisers of $G$ yields a group which is fully residually $G$. For further details, we refer the reader to \cite[Section 3]{raags}. Free groups, free abelian groups and more generally coherent RAAGs all lie in the class $\mathcal{C}$; in addition, toral relatively hyperbolic groups (torsion-free groups which are hyperbolic relative to a set of free abelian groups) are also in $\mathcal{C}$.

\subsection{Exponential groups}
\label{Agroup}
Recall that throughout this paper we assume that rings are associative, have a free abelian additive group and a multiplicative identity $1$; as a consequence, characteristic subrings are always isomorphic to $\mathbb{Z}$. The definitions appearing below are simplified versions of the originals: for every definition which involves a ring $A$ and a subring $A_0$, we assume that $A_0=\mathrm{char}(A)\cong\mathbb{Z}$. Further detail can be found in \cite{exp1} and \cite{exp}.

\begin{definition}
	Let $A$ be a ring. A group $G$ is called an \emph{$A$-group} if there is a map $G\times A\rightarrow G$ which satisfies the following (below, $g^a$ denotes the image of $(g,a)$ under the map $G\times A \rightarrow G$):
	\begin{enumerate}
		\item $g^1=g$, $g^0=1$ and $1^a=1$ for every $g\in G$ and $a\in A$,
		\item $g^{a+b}=g^ag^b$ and $(g^a)^b=g^{ab}$ for every $g\in G$ and $a,b\in A$,
		\item $(hgh^{-1})^a=hg^ah^{-1}$ for every $g,h\in G$ and $a\in A$, and
		\item for every $g,h\in G$ and $a\in A$, if $[g,h]=1$ then $(gh)^a=g^ah^a$.
	\end{enumerate}
	We call $G$ a \emph{partial $A$-group} if there exists $P\subset G\times A$ such that $g^a$ is defined whenever $(g,a)\in P$, and all the properties above hold whenever the arguments belong to $P$.
\end{definition}
\par 
\medskip
A homomorphism $f:G\rightarrow H$ where $G$ and $H$ are $A$-groups is called an \emph{$A$-homomorphism} if $f(g^a)=(f(g))^a$ for every $g\in G$ and $a\in A$. If $H\le G$ and $G$ is a partial $A$-group we say that $H$ is a \emph{full $A$-subgroup} of $G$ if $h^a$ is defined, and lies in $H$, for every $h\in H$ and $a\in A$. Within our limited settings, an $A$-completion of a group $G$ is defined as follows:
\begin{definition}
	Let $G$ be a group. An $A$-completion of $G$ is an $A$-group $G^A$ which satisfies the following:
	\begin{enumerate}
		\item there is a homomorphism $\tau:G\rightarrow G^A$ such that no proper full $A$-subgroup of $G^A$ contains $\tau(G)$, and
		\item if $f:G\rightarrow H$ is a homomorphism and $H$ is an $A$-group, then $f$ factors via $G^A$; in other words, there exists a unique $A$-homomorphism $\overline{f}:G^A\rightarrow H$ such that $f=\overline{f}\circ\tau$.
	\end{enumerate}
\end{definition}
\par 
\medskip
By \cite[Theorems 1 and 2]{exp1}, every group admits an $A$-completion and this completion is unique (up to $A$-isomorphism). We also remark that if $G$ is abelian, then $G^A$ is abelian and coincides with $G\otimes_{\mathbb{Z}}A$ as the two groups satisfy the same universal property.
	
\subsection{Limit groups}
\label{limitgps}

Limit groups (over free groups) were first defined by Sela in \cite{sela}; this class of groups coincides with the class of finitely generated fully residually free groups, which has been extensively studied since the 1960s. In this paper we deal with limit groups over groups in the class $\mathcal{C}$, and for expository purposes, throughout this subsection, adopt the approach of Champetier and Guirardel (see \cite{limitgps}).
\par 
\medskip
Given a group $K$, a \emph{limit group} over $K$ is, simply put, a limit of finitely generated marked subgroups of $K$ in the space of marked groups. A \emph{marked group} is a pair $(G,S)$ such that $G$ is a group and $S$ is a finite generating set of $G$. We say that two marked groups $(G,\{s_1,\ldots,s_n\})$ and $(G',\{s'_1,\ldots,s'_n\})$ are isomorphic (as marked groups) if the map which sends each $s_i\in G$ to $s'_i\in G'$ extends to an isomorphism $G\rightarrow G'$. Fixing a positive integer $n$, we define $\mathcal{G}_n$ to be the set of marked groups $(G,S)$ such that $\vert S\vert =n$. 
\par 
\medskip
The set $\mathcal{G}_n$ can be viewed as a topological space, where the topology is induced by the following pseudometric: given $(G,S),(G',S')\in \mathcal{G}_n$, set $v((G,S),(G',S'))$ to be the maximal integer $N$ such that $w(S)=1$ in $G$ if and only if $w(S')=1$ in $G'$ for every word $w$ of length at most $N$. If $(G,S)$ and $(G',S')$ are isomorphic as marked groups, set $v((G,S),(G',S'))=\infty$. The distance between $(G,S)$ and $(G',S')$ in $\mathcal{G}_n$ is
\begin{equation*}
	d_n((G,S),(G',S'))=e^{-v((G,S),(G',S'))}.
\end{equation*}
\begin{definition}
	Let $G$ be a group. A group $H$ is called a \emph{limit group over $G$} if there is an integer $n$ such that $H$ is the limit of a sequence of marked groups $(H_i,S_i)\in \mathcal{G}_n$ (with respect to the topology defined above), and such that $H_i\le G$ for every $i\in \mathbb{N}$.
\end{definition}
\par 
\medskip
In many cases, and in particular the cases that interest us (namely coherent RAAGs and toral relatively hyperbolic groups), limit groups admit a simpler description. A group $G$ is called \emph{equationally Noetherian} if the following holds: given a tuple $\textbf{x}=(x_1,\ldots,x_n)$ of variables, and a set $\Sigma\subset G\ast F(\textbf{x})$ of \emph{equations over $G$}, there is a finite subset $\Sigma_0$ of $\Sigma$ such that
\begin{equation*}
\{\textbf{g}=(g_1,\ldots,g_n)\in G^n \vert \sigma(\textbf{g})=1\;\;\forall \sigma\in \Sigma\}=\{\textbf{g}=(g_1,\ldots,g_n)\in G^n \vert \sigma(\textbf{g})=1\;\;\forall \sigma\in \Sigma_0\}
\end{equation*}
where $\sigma(\textbf{g})$ is the element of $G$ obtained by replacing each $x_i$ with $g_i$. In other words, every system of equations over $G$ is equivalent to a finite subsystem. If $G$ is equationally Noetherian, then $H$ is a limit group over $G$ if and only if $H$ is finitely generated and fully residually $G$. The fact that every finitely generated fully residually $G$ group is a limit group is easy to see, and is true for any group. We sketch a proof for the converse below:
\par 
\medskip
Suppose that $(H,S)$ is the limit of a sequence $(H_i,S_i)$ in $\mathcal{G}_n$ and that $H_i\le G$ for every $i\in \mathbb{N}$. In the free group $F(S)$, let $\Sigma \subset F(S)$ be the kernel of the homomorphism $F(S)\rightarrow H$ induced by the inclusion of $S$. Since $G$ is equationally Noetherian, the system of equations $\Sigma$ is equivalent to a finite subsystem $\Sigma_0\subset \Sigma$ in $G$. As $\sigma(S)=1$ in $H$ for every $\sigma \in \Sigma_0$, $\sigma(S_i)=1$ in $H_i$ for sufficiently large $i$. This implies that the map which sends $S$ to $S_i$ extends to a homomorphism $f_i:H\rightarrow H_i$ for sufficiently large $i$. Finally, for any finite subset $E$ of $H$, $(H_i,S_i)$ is sufficiently close to $(H,S)$ in $\mathcal{G}_n$ (as long as $i$ is large enough) to ensure that $f_i:H\rightarrow H_i$ is injective on $E$.
\par 
\medskip
Groves showed in \cite[Theorem 5.16]{groves} that toral relatively hyperbolic groups are equationally Noetherian. The fact that coherent RAAGs are equationally Noetherian follows from their linearity and is mentioned in \cite{limitraags}.
\par 
\medskip
If $G$ is a coherent RAAG or a toral relatively hyperbolic group, then limit groups over $G$ admit yet another description: they are the finitely generated subgroups of the $\mathbb{Z}[t]$-completion of $G$ (see \cite[Corollary 6.12 and Theorem 8.1]{raags} and \cite[Theorems D. and E.]{torallim}). We further explore this characterisation of limit groups in the following section.

\section{$A$-completions and cyclic subgroup separability}
\label{mainsec}

Exponential completions of (certain) groups in the class $\mathcal{C}$ exhibit a fairly friendly structure: they can be built, "from the group $G$ up", by iterating \emph{extensions of centralisers}.

\begin{definition}
	Let $G$ be a group and let $u\in G$. Let $H$ be another group, and let $\varphi:C_G(u)\rightarrow H$ be an injective homomorphism such that $\varphi(u)\in Z(H)$. The \emph{extension of the centraliser $C_G(u)$ by $H$} is the group
	\begin{equation*}
		G(u,H)=G\ast_{C_G(u)=\varphi(C_G(u))}H.
	\end{equation*}
	If $\varphi(C_G(u))$ is a direct factor of $H$, then the extension is said to be \emph{direct}. If, furthermore, $H=\varphi(C_G(u))\times \mathbb{Z}$, the extension is said to be \emph{free}.
\end{definition}
\par
\medskip
Direct extensions of centralisers have a particularly nice structure. If $G(u,C_G(u)\times B)$ is a direct extension of the centraliser $C_G(u)$ by $C=C_G(u)\times B$, then
\begin{align*}
	G(u,C_G(u)\times B) & = G \ast _{C_G(u)} (C_G(u)\times B) \\
	& = G \ast _{C_G(u)} \langle C_G(u),B \vert \, [C_G(u),B]=1 \rangle \\
	& = \langle G,B \vert [C_G(u),B]=1 \rangle,
\end{align*}
or in other words $G(u,C_G(u)\times B)$ is the free product of $G$ and $B$ with commuting subgroups $C_G(u)$ and $B$.
\par
\medskip
In \cite{loginovacyc}, the author gives a criterion under which free products with commuting subgroups are cyclic subgroup separable. This was later generalized in \cite{sokolov}. 

\begin{theorem}[{\cite[Main Theorem]{loginovacyc}, \cite[Theorem 2.1]{sokolov}}]
	\label{cyc}
	Let $G$ and $H$ be cyclic subgroup separable groups and let $L\le G$ and $M\le H$. If the group with presentation $\langle G,H \vert \; [L,M]=1 \rangle$ is residually finite, then it is cyclic subgroup separable.
\end{theorem}

\par 
\medskip
This criterion will be used throughout this section; as a warm-up we prove the following lemma which easily follows from Theorem \ref{cyc}:

\begin{lemma}
	\label{direct}
	Let $G$ be a group in the class $\mathcal{C}$ and let $u\in G$ be such that $C_G(u)$ is abelian. Let $B$ be a free abelian group and write $C=C_G(u)\times B$. If $G$ is cyclic subgroup separable, then so is the direct centraliser extension $G(u,C)$.
\end{lemma}

\begin{proof}
	From the discussion preceding this lemma, $G(u,C)$ is the free product of $G$ and $B$ with commuting subgroups $C_G(u)$ and $B$. Since $B$ is free abelian, its cyclic subgroups are separable; by our assumption, $G$ is also cyclic subgroup separable. Therefore, by Theorem \ref{cyc}, it is enough to show that $G(u,C)$ is residually finite.
	\par
	\medskip
	By \cite[Theorem 4.2]{raags}, $G(u,C)$ is fully residually $G$. Since $G$ is cyclic subgroup separable it is residually finite, and hence $G(u,C)$ is residually finite.
\end{proof}

\begin{remark}
	It is worth mentioning that in an earlier paper, Loginova shows that a free product with commuting subgroups $\langle G,H \vert \; [L,M]=1 \rangle$ is residually finite if and only if $G$ and $H$ are residually finite, $L$ is separable in $G$ and $M$ is separable in $H$ (see \cite[Theorem 1]{loginova}). It follows that if $G$ is a residually finite group in the class $\mathcal{C}$, then abelian centralisers in $G$ are separable. In particular, abelian centralisers in graph towers over coherent RAAGs (see \cite[Section 7]{raags}) are separable.
\end{remark}
\par 
\medskip
The remainder of this section is devoted to proving Theorems \ref{main} and \ref{css}. We do so by analysing the construction of the $A$-completion of a group $G$ from the class $\mathcal{C}$ in steps, following \cite{raags}, and proving that each step yields a cyclic subgroup separable group. Note that by \cite[Proposition 6.1]{raags}, if $G$ is abelian then $G^A$ is abelian, and therefore cyclic subgroup separable. We therefore restrict our attention to non-abelian groups in the class $\mathcal{C}$. We also remark that in \cite{raags}, the authors assume that a non-abelian group $G\in \mathcal{C}$ satisfies an additional condition, named \emph{condition $\mathrm{R}$}, in order to show that $G^A$ enjoys the structure of an iterated centraliser extension. We make this assumption too.

\begin{definition}
	\label{condR}
	A group $G\in \mathcal{C}$ is said to satisfy \emph{condition $\mathrm{R}$} if it is a partial $A$-group, and for every $u\in G$, if $C_G(u)$ is non-abelian, then the centre $Z(C_G(u))$ of $C_G(u)$ is a full $A$-subgroup.
\end{definition}

We remind the reader that throughout this paper we assume that all rings are associative, have a free abelian additive subgroup and a multiplicative identity $1$. We also recall the statement of Theorem \ref{main} for the convenience of the reader:

\begin{namedtheorem}[Theorem 1]
	Let $G$ be a group in the class $\mathcal{C}$ which satisfies condition $\mathrm{R}$, and let $A$ be a ring. If $G$ is cyclic subgroup separable, then the $A$-completion of $G$, $G^A$, is cyclic subgroup separable.
\end{namedtheorem}

The strategy behind the construction of the $A$-completion of $G$ is rather straightforward: we repeatedly extend centralisers in $G$ to obtain a group $G^*$ such that $G$ is a full $A$-subgroup of $G^*$; in other words, for every $g\in G$, the action of $A$ on $g$ within $G^*$ is defined. Iterating this construction, we eventually obtain the $A$-completion of $G$. We recall the following construction from \cite{exp} which allows us to extend multiple centralisers at once:

\begin{definition}[{\cite[Definition 8]{exp}, \cite[Definition 4.7]{raags}}]
	\label{treedef}
	Let $\mathfrak{C}=\{C_G(u_i)\}_{i\in I}$ be a set of centralisers in a group $G$ and let $\{\varphi_i:C_G(u_i)\rightarrow H_i\}_{i\in I}$ be injective homomorphisms such that $\varphi_i(u_i)\in Z(H_i)$ for every $i\in I$. Let $T$ be the tree whose vertex set is $\{v\}\cup \{v_i\}_{i\in I}$ and whose edge set is $\{e_i=(v,v_i)\}_{i \in I}$. Let $T_G$ be the graph of groups whose underlying graph is $T$, and whose vertex groups, edge groups and edge maps are as follows:
	\begin{enumerate}
		\item $G_v=G$,
		\item $G_{v_i}=H_i$,
		\item $G_{e_i}=C_G(u_i)$,
		\item the map which maps $G_{e_i}$ into $G_v$ is the inclusion, and
		\item the map which maps $G_{e_i}$ into $G_{v_i}$ is $\varphi_i$.
	\end{enumerate}
	The fundamental group of this graph of groups is called a \emph{tree extension of centralisers}. It is denoted by $G(\mathfrak{C},\mathcal{H},\Phi)$ where $\mathcal{H}=\{H_i\}_{i\in I}$ and $\Phi=\{\varphi_i\}_{i\in I}$.
\end{definition}
\par 
\medskip
We have already seen (Lemma \ref{direct}) that direct extensions of abelian centralisers of cyclic subgroup separable groups in $\mathcal{C}$ are cyclic subgroup separable. The following lemma, which relies on \cite[Proposition 4.8]{raags}, shows that the same holds for tree extensions of centralisers:

\begin{lemma}
	\label{tree}
	Let $G$ be a group in the class $\mathcal{C}$ and let $\mathfrak{C}=\{C_G(u_i)\}_{i\in I}$ be a set of abelian centralisers in $G$ such that no two of them are conjugate. For each $i\in I$, let $H_i$ be a free abelian group and let $\varphi_i:H_i\rightarrow C_G(u_i)$ be an injective homomorphism such that $H_i=\varphi_i(C_i)\times K_i$ for some $K_i\le H_i$. \\ Keeping the notation of Definition \ref{treedef}, if $G$ is cyclic subgroup separable, then the tree extension of centralisers $G(\mathfrak{C},\mathcal{H},\Phi)$ is cyclic subgroup separable.
\end{lemma}

\begin{proof}
	Let $<$ be a well-ordering of the set $I\cup\{0\}$ (assuming $0\notin I$ and $0<i$ for every $i\in I$). We construct, by recursion, a direct system of groups $\{G_i\}_{i\in I\cup \{0\}}$ over $I$, along with inclusion maps $f_{i,j}:G_i\rightarrow G_j$ (for $i<j$) and retractions $r_{j,i}:G_j\rightarrow G_i$ (for $i<j$). For readability, we refer to the maps $r_{j,i}$ as retractions, but formally we mean that $r_{j,i}(f_{i,j}(g))=g$ for every $g\in G_i$. In addition,
	\begin{enumerate}
		\item $G_0=G$,
		\item for every $j<i\in I$, the centraliser of $u_{i}$ in $G_j$ coincides with the centraliser of $u_{i}$ in $G$, that is $C_{G_j}(f_{0,j}(u_{i}))=f_{0,j}(C_G(u_{i}))$,
		\item every $G_i$ is cyclic subgroup separable,
		\item $\bigcup_{i\in I}G_i=G(\mathfrak{C},\mathcal{H},\Phi)$.
	\end{enumerate}
	Suppose first that $i\in I$ is a successor ordinal, that is $i=j+1$; suppose in addition that for every $k\le j$ the groups $G_k$ have been defined, along with the suitable inclusion maps and retractions. Set $G_i$ to be the direct extension of the centraliser $C_{G_j}(f_{0,j}(u_{i}))$ by $H_i$. Note that this is well-defined, since we assume that $C_{G_j}(f_{0,j}(u_i))=f_{0,j}(C_G(u_i))$ so $C_{G_j}(f_{0,j}(u_i))$ is a direct factor of $H_i$. By Lemma \ref{direct}, $G_i$ is cyclic subgroup separable. Let $f_{j,i}$ be the obvious inclusion map $G_j\rightarrow G_i$, and for every $k<j$ set $f_{k,i}=f_{j,i}\circ f_{k,j}$. Define the retraction $r_i,j:G_i\rightarrow G_j$ by mapping every element of $K_i$ to $f_{0,j}(u_i)$. Similarly, for every $k\le j$ set $r_{i,k}=r_{j,k}\circ r_{i,j}$. In addition, for every $k>i$, the fact that $f_{0,i}(u_k)$ is not a conjugate of any element in $C_{G_j}(f_{0,j}(u_i))$ implies that $C_{G_i}(f_{0,i}(u_k))=f_{0,i}(C_G(u_k))$. 
	\par
	\medskip
	Suppose now that $i$ is a limit ordinal, and that for every $k<i$ the groups $G_k$, along with the suitable inclusion maps and retractions, have been defined. Consider the directed system of groups $\{G_j\}_{j<i}$ and let $\overline{G}_i$ be its direct limit. Denote by $\overline{f}_{j,i}:G_j\rightarrow \overline{G}_i$ the canonical embedding of $G_j$ in $\overline{G}_i$ for $j<i$. To define retractions $\overline{r}_{i,j}:\overline{G}_i\rightarrow G_j$, consider the cofinal system $\{G_k\}_{j<k<i}$; its direct limit is $\overline{G}_i$. For every $j<k\le \ell<i$ we have the following commuting diagram:
	\begin{equation*}
	\xymatrix{
	G_k \ar[dr]^{r_{k,j}} \ar[d]_{f_{k,\ell}} & \\
	G_\ell \ar[r]_{r_{\ell,j}} & G_j
	}
	\end{equation*}
	and by the universal property of direct limits we obtain a map $\overline{r}_{i,j}:\overline{G}_i\rightarrow G_j$ along with the following commuting diagrams (for $j<k<i$):
	\begin{equation*}
	\xymatrix{
		G_k \ar[r]^{\overline{f}_{k,i}} \ar[dr]_{r_{k,j}} & \overline{G}_i \ar[d] ^{\overline{r}_{i,j}} \\
		& G_j
	}
	\end{equation*}
	Note that for $g\in G_j$,
	\begin{align*}
		\overline{r}_{i,j}(\overline{f}_{j,i}(g))&=\overline{r}_{i,j}\circ \overline{f}_{k,i} (f_{j,k}(g)) \\
		& =r_{k,j}(f_{j,k}(g))\\
		&=g
	\end{align*}
	so $\overline{r}_{i,j}$ is a retraction. In addition, the fact that
	\begin{equation*}
	C_{G_j}(f_{0,j}(u_k))=f_{0,j}(C_G(u_k))
	\end{equation*}
	for every $j<i$ and $k\ge i$ implies that $C_{\overline{G}_i}(\overline{f}_{0,i}(u_k))=\overline{f}_{0,i}(C_G(u_k))$ for every $k \ge i$. The existence of these retractions also implies that $\overline{G}_i$ is cyclic subgroup separable. Let $g,h\in \overline{G}_i$ be such that $h\notin \langle g \rangle$; there is $j<i$ and $g',h' \in G_j$ such that $\overline{f}_{j,i}(g')=g$, $\overline{f}_{j,i}(h')=h$ and $h' \notin \langle g' \rangle$. Since $G_j$ is cyclic subgroup separable, there is a map $q:G_j\rightarrow Q$ such that $q(h') \notin \langle q(g') \rangle$ and $Q$ is finite. The map $q \circ \overline{r}_{i,j}: \overline{G}_i \rightarrow Q$ separates $h$ from $\langle g \rangle$.
	\par
	\medskip
	We now define $G_i$ to be the direct extension of the centraliser $C_{\overline{G}_i}(\overline{f}_{0,i}(u_i))$ by $H_i$. By Lemma \ref{direct}, $G_i$ is cyclic subgroup separable. We also set $\overline{f}_i:\overline{G}_i\rightarrow G_i$ to be the inclusion map, and $\overline{r}_i:G_i\rightarrow \overline{G}_i$ to be the retraction which maps $K_i$ to $\overline{f}_{0,i}(u_i)$. Finally, define $f_{j,i}=\overline{f}_i\circ \overline{f}_{j,i}$ and $r_{i,j}=\overline{r}_{i,j}\circ \overline{r}_i$. The desired properties of $G_i$, the maps $f_{j,i}$ and the retractions $r_{i,j}$ can be verified as in the successor stage.
	\par
	\medskip
	The cyclic subgroup separability of $\bigcup_{i\in I}G_i=G(\mathfrak{C},\mathcal{H},\Phi)$ also follows, as in either the successor or the limit stage, depending on the order type of $\{0\}\cup I$.
\end{proof}

With Lemma \ref{tree} in our arsenal, we are ready to describe the construction of the $A$-completion of a group $G$ in $\mathcal{C}$ and prove Theorem \ref{main}. Assume in addition that $G$ satisfies condition $\mathrm{R}$; recall that, as in the paragraph preceding Theorem \ref{main}, we first construct a group $G^\ast$ which contains $G$, and such that $G$ is a full $A$-subgroup of $G^*$. We begin by choosing a set of centralisers $\mathfrak{C}(G)=\{C_G(u_i)\}_{i\in I}$ in $G$ which satisfies the following:
\begin{enumerate}
	\item every centraliser in $\mathfrak{C}$ is abelian, and \emph{not} a full $A$-subgroup of $G$,
	\item no two centralisers in $\mathfrak{C}(G)$ are conjugate, and
	\item any abelian centraliser in $G$ which is not a full $A$-subgroup is conjugate to a centraliser in $\mathfrak{C}(G)$.
\end{enumerate}
Note that the existence of a set $\mathfrak{C}$ which satisfies the conditions above is guaranteed by Zorn's Lemma. Recall that as in Subsection \ref{Agroup}, for every $C_G(u_i)\in \mathfrak{C}(G)$ we have that $C_G(u_i)^A=C_G(u_i)\otimes_\mathbb{Z}A$; in addition, $C_G(u_i)$ is a direct summand of $C_G(u_i)^A$. Setting 
\begin{equation*}
	\mathcal{H}(G)=\{C_G(u_i)^A\}_{i\in I}
\end{equation*}
and 
\begin{equation*}
	\Phi(G)=\{\varphi_i:C_G(u_i)\rightarrow C_G(u_i)^A\}_{i\in I}
\end{equation*}
where each $\varphi_i$ is the canonical embedding, we define
\begin{equation*}
	G^\ast = G(\mathfrak{C}(G),\mathcal{H}(G),\Phi(G)).
\end{equation*}
By \cite[Lemma 6.5]{raags} $G$ is a full $A$-subgroup of $G^\ast$. In addition, by \cite[Lemma 6.6]{raags}, $G^\ast$ satisfies condition $\mathrm{R}$ and we can iterate this construction. As in \cite[Subsection 6.2]{raags}, we define a directed system of groups
\begin{equation*}
	G=G^{(0)}<G^{(1)}<\cdots < G^{(n)}<\cdots
\end{equation*}
where
\begin{equation*}
	G^{(n+1)}=(G^{(n)})^\ast=G^{(n)}(\mathfrak{C}(G^{(n)}),\mathcal{H}(G^{(n)}),\Phi(G^{(n)}))
\end{equation*}
and the maps $f_{i,j}:G^{(i)} \rightarrow G^{(j)}$ are the inclusion maps. The direct limit of this system $\bigcup_{n \in \mathbb{N}}G^{(n)}$ is called an \emph{iterated centraliser extension} of $G$ by $A$, or in short an ICE of $G$ by $A$. Note that $\bigcup_{n \in \mathbb{N}}G^{(n)}$ is an $A$-group, since every $g\in G$ lies in $G^{(n)}$ for some $n$, and therefore the action of $A$ on $g$ is already defined in $G^{(n+1)}$. As a matter of fact, $\bigcup_{n \in \mathbb{N}}G^{(n)}$ is the $A$-completion of $G$ as evident in \cite[Theorem 6.3]{raags}. Theorem \ref{main} now follows:

\begin{proof}[Proof of Theorem \ref{main}]
	By Lemma \ref{tree}, $G^{(n+1)}$ retracts onto $G^{(n)}$ for every $n\in \mathbb{N}$; composing these retractions we obtain retractions $r_{n,m}:G^{(n)}\rightarrow G^{(m)}$ for every $m<n$. As in the proof of Lemma \ref{tree}, these retractions imply the existence of retractions from the direct limit $\bigcup_{n \in \mathbb{N}}G^{(n)}$ onto each $G^{(n)}$. In addition, each $G^{(n)}$ is cyclic subgroup separable.
	\par 
	\medskip
	Let $g,h\in \bigcup_{n \in \mathbb{N}}G^{(n)}$ be such that $h\notin \langle g \rangle$; $g$ and $h$ lie in some $G^{(n)}$. Since $G^{(n)}$ is cyclic subgroup separable, there is a homomorphism $q:G^{(n)}\rightarrow Q$ such that $q(h)\notin \langle q(g) \rangle$ and $Q$ is finite. The composition $q\circ r_n: \bigcup_{n \in \mathbb{N}}G^{(n)}\rightarrow Q$ separates $h$ from $\langle g \rangle$. 
\end{proof}

\begin{cor}
	Limit groups over cyclic subgroup separable toral relatively hyperbolic groups are cyclic subgroup separable.
\end{cor}

\begin{proof}
	Let $G$ be a toral relatively hyperbolic group; in particular $G$ lies in $\mathcal{C}$ and satisfies condition $\mathrm{R}$. By Theorem \ref{main}, the $\mathbb{Z}[t]$-completion of $G$ is cyclic subgroup separable, and by \cite[Theorems D. and E.]{torallim} limit groups over $G$ are exactly the finitely generated subgroups of $G^{\mathbb{Z}[t]}$.
\end{proof}
\par 
\medskip
With a bit more work, we can also deduce the following:
\begin{namedtheorem}[Theorem 2]
	Limit groups over coherent RAAGs are cyclic subgroup separable.
\end{namedtheorem}

\begin{proof}
	Let $G(\Gamma)$ be a coherent RAAG. By \cite[Corollary 6.12 and Theorem 8.1]{raags}, limit groups over $G(\Gamma)$ are exactly the finitely generated subgroups of $G(\Gamma,\mathbb{Z}[t])^{\mathbb{Z}[t]}$ (where $G(\Gamma,\mathbb{Z}[t])$ is the graph product whose underlying graph is $\Gamma$, and whose vertex groups are all $\mathbb{Z}[t]$). In light of Theorem \ref{main}, since $G(\Gamma,\mathbb{Z}[t])$ lies in $\mathcal{C}$ and satisfies condition $\mathrm{R}$, it is enough to show that $G(\Gamma,\mathbb{Z}[t])$ is cyclic subgroup separable. Let $g,h\in G(\Gamma,\mathbb{Z}[t])$ be such that $h\notin \langle g \rangle$; there is a finite full subgraph $\Delta$ of $\Gamma$ such that $g,h\in G(\Delta,\mathbb{Z}[t])$. Note that $G(\Gamma,\mathbb{Z}[t])$ retracts onto $G(\Delta,\mathbb{Z}[t])$ by killing each vertex group $G_v$ for $v\notin \mathrm{V}\Delta$. Hence it is sufficient to show that $G(\Delta,\mathbb{Z}[t])$ is cyclic subgroup separable for every finite full subgraph $\Delta$ of $\Gamma$. 
	\par
	\medskip
	Since $\Gamma$ is chordal, so is $\Delta$. It is a famous result that every finite chordal graph, and in particular $\Delta$, admits a \emph{perfect elimination ordering} (see \cite{chordal}), that is an ordering $(v_1,v_2,\ldots,v_n)$ of $\mathrm{V}\Delta$ such that the following holds: for every $i$, the neighbours of $v_i$ amongst $v_1,v_2,\ldots,v_{i-1}$ form a clique. We show by induction on $n$ that $G(\Delta, \mathbb{Z}[t])$ is cyclic subgroup separable. For every $i\le n$ denote by $\Delta_i$ the full subgraph of $\Delta$ whose vertices are $v_1,v_2,\ldots,v_i$; denote by $G_i$ the copy of $\mathbb{Z}[t]$ which corresponds to the vertex $v_i$ of $\Delta$. For $n=1$, $G(\Delta_1,\mathbb{Z}[t])$ is free abelian and therefore cyclic subgroup separable. Suppose now that $G(\Delta_i,\mathbb{Z}[t])$ is cyclic subgroup separable and that the neighbours of $v_{i+1}$ in $\Delta$ are $v_{i_1},\ldots,v_{i_k}$. We have that
	\begin{align*}
		G(\Delta_{i+1},\mathbb{Z}[t])&=\langle G, G_{i+1} \vert \; [G_{i+1},G_{i_j}],\; j=1,\ldots,k \rangle \\
		& =\langle G, G_{i+1} \vert \; [G_{i+1},\langle G_{i_1},\ldots,G_{i_k}\rangle] \rangle
	\end{align*}
	where the last equality follows from the fact that $v_{i_1},\ldots,v_{i_k}$ form a clique in $\Delta_i$ and therefore $\langle G_{i_1},\ldots,G_{i_k}\rangle=G_{i_1}\times \cdots \times G_{i_k}$. In other words, $G(\Delta_{i+1},\mathbb{Z}[t])$ is the free product of $G(\Delta_i,\mathbb{Z}[t])$ and $G_{i+1}$ with commuting subgroups $\langle G_{i_1},\ldots,G_{i_k}\rangle$ and $G_{i+1}$. Since $G(\Delta_i,\mathbb{Z}[t])$ and $G_{i+1}$ are cyclic subgroup separable and since $G(\Delta_{i+1},\mathbb{Z}[t])$ is residually finite as the graph product of residually finite groups (see, for example, \cite{graphprod}), it follows that $G(\Delta_i,\mathbb{Z}[t])$ is cyclic subgroup separable by Theorem \ref{cyc}.
\end{proof}

\section{Free products with commuting subgroups and the word problem}

It is well-known that an amalgamated product $G\ast_K H$ admits a solution to the word problem if the word problem is solvable in $G$ and in $H$ and there is a solution to the membership problem for $K$ in both $G$ and $H$. A similar statement can be made for free products with commuting subgroups:
\begin{lemma}
	\label{word}
	Let $G$ and $H$ be groups with a solvable word problem and let $L\le G$ and $M\le H$. Suppose that the membership problem is solvable for $L$ in $G$ and for $M$ in $H$. Then there is a solution to the word problem in $\langle G,H \vert \; [L,M]=1\rangle$.
\end{lemma}

Recall that a \emph{free centraliser extension} is a centraliser extension of the form $G(u,C_G(u)\times\mathbb{Z})=\langle G,t \vert \; [C_G(u),t]=1\rangle$,
where $u\in G$. Using Lemma \ref{word} above, we obtain:

\begin{namedprop}[Proposition 3]
	Let $G$ be a group in the class $\mathcal{C}$. If $G$ satisfies condition $\mathrm{R}$ and has a solvable word problem, then every finitely generated subgroup $H$ of $G^A$ has a solvable word problem.
\end{namedprop}

\begin{proof}
	The fact that $H$ is finitely generated and embeds in $G^A$ implies that $H$ embeds in a group obtained from $G$ by taking finitely many free extensions of centralisers; that is, there are groups $G_0,G_1,\ldots,G_n$ such that $G_0=G$, $G_{i+1}=\langle G_i, t_i \vert \; [C_{G_i}(u_i),t_i]=1\rangle$ for some $u_i \in G_i$ and $H\le G_n$. 
	\par
	\medskip
	We prove that $G_n$, and hence $H$, has a decidable word problem by induction on $n$. Suppose that $G_i$ has a solvable word problem. By Lemma \ref{word}, a solution to the membership problem for $C_{G_i}(u_i)$ in $G_i$ would imply that $G_{i+1}$ has a solvable word problem. But checking whether $g\in G_i$ lies in $C_{G_i}(u_i)$ is equivalent to asking whether $[g,u_i]=1$ in $G_i$, which is solvable by the induction hypothesis. Hence the word problem in $G_{i+1}$ is solvable, which completes the proof.
\end{proof}

\begin{cor}
	Limit groups over coherent RAAGs and toral relatively hyperbolic groups have a solvable word problem.
\end{cor}

\begin{proof}
	If $G(\Gamma)$ is a coherent RAAG and $H$ is a limit group over $G(\Gamma)$, then by \cite[Theorem 8.1]{raags} $H$ is a finitely generated subgroup of $G(\Gamma,\mathbb{Z}[t])^{\mathbb{Z}[t]}$. The group $G(\Gamma,\mathbb{Z}[t])$ lies in $\mathcal{C}$, satisfies condition $\mathrm{R}$ and admits an algorithm which checks whether a given word in the canonical generators is trivial or not. Similarly, if $G$ is toral relatively hyperbolic and $H$ is a limit group over $G$, then by \cite[Theorems D. and E.]{torallim} $H$ is a finitely generated subgroup of $G^{\mathbb{Z}[t]}$. $G$ satisfies condition $\mathrm{R}$ and by \cite[Subsection 2.7]{toralword} has a solvable word problem.
\end{proof}

The fact that limit groups over coherent RAAGs have a decidable word problem was already mentioned in \cite{raags}, and follows from these groups being finitely presented and residually finite. The following proposition, which we record here for the sake of completeness, gives a solution to the word problem for limit groups over toral relatively hyperbolic groups (toral relatively hyperbolic groups are equationally Noetherian \cite[Theorem 5.16]{groves} and limit groups over toral relatively hyperbolic groups are recursively presented since they embed in finitely presented groups):

\begin{prop}
	Let $G$ be an equationally Noetherian group. If $G$ has a solvable word problem, then so does every finitely generated, recursively presented, residually $G$ group.
\end{prop}

\begin{proof}
	Let $H$ be a finitely generated, recursively presented, residually $G$ group. We execute the following two algorithms in parallel: first, since $H$ is recursively presented there is an algorithm which takes a word $g\in H$ as its input, and returns 'yes' if $g=1$. 
	\par
	\medskip
	Second, let $S$ be a finite generating set of $H$ and let $g$ be a word in the alphabet $S \cup S^{-1}$. The equational Noetherianity of $G$ implies that  there is an algorithm which checks, within finite time, whether a map $S\rightarrow G$ extends to a homomorphism $f:H\rightarrow G$. Given such a homomorphism $f$, using a solution to the word problem in $G$ the algorithm can further verify whether or not $f(g)\ne 1$. Since $H$ is residually $G$, the algorithm described will return 'no' whenever $g\ne 1$ in $H$.
\end{proof}

\bibliographystyle{amsalpha}
\providecommand{\bysame}{\leavevmode\hbox to3em{\hrulefill}\thinspace}
\providecommand{\MR}{\relax\ifhmode\unskip\space\fi MR }
\providecommand{\MRhref}[2]{%
  \href{http://www.ams.org/mathscinet-getitem?mr=#1}{#2}
}
\providecommand{\href}[2]{#2}

\par 
\medskip 

\textsc{Mathematical Institute, University of Oxford, Oxford OX1 3LB, England}
\par
\smallskip
\textit{E-mail address}: \texttt{fruchter@maths.ox.ac.uk}

\end{document}